\documentclass[times]{nlaauth}

\usepackage{moreverb}

\usepackage[colorlinks,bookmarksopen,bookmarksnumbered,citecolor=red,urlcolor=red]{hyperref}

\newcommand\BibTeX{{\rmfamily B\kern-.05em \textsc{i\kern-.025em b}\kern-.08em
T\kern-.1667em\lower.7ex\hbox{E}\kern-.125emX}}

\def\volumeyear{2017}

\usepackage[english]{babel}                     	
\usepackage[utf8]{inputenc}
\usepackage{amsmath,amsfonts,amssymb,amsthm,amsxtra} 		
\usepackage{bbm}                    					 	
\usepackage{graphicx}               					 	
\usepackage{caption}										
\usepackage{subcaption}										
\usepackage{algpseudocode}									
\usepackage{framed}											
\usepackage{mathtools}										
\usepackage{tikz}											
\usepackage{bm}												
\usepackage{upgreek}										
\usepackage{enumitem}										
\usepackage{chngcntr}										
\usepackage{diagbox}										
\usepackage{array}											
\usepackage[section]{algorithm}								

\numberwithin{equation}{section}

\usepackage[colorlinks,bookmarksopen,bookmarksnumbered,citecolor=red,urlcolor=red]{hyperref}

\theoremstyle{plain}

\newtheorem{lemma}[algorithm]{Lemma}
\newtheorem{corollary}[algorithm]{Corollary}

\theoremstyle{definition}
\newtheorem{definition}[algorithm]{Definition}

\theoremstyle{remark}
\newtheorem{remark}[algorithm]{Remark}

\newcommand{\norm}[1]{\lVert#1\rVert} 							
\newcommand{\normb}[1]{\big\lVert#1\big\rVert} 					
\newcommand{\dotprod}[2]{\langle#1,#2\rangle} 					
\newcommand{\dotprodb}[2]{\big\langle#1,#2\big\rangle} 			
\DeclareMathOperator{\rank}{rank}								
\DeclareMathOperator{\grad}{grad}								
\DeclareMathOperator{\tr}{tr}								
\DeclareMathOperator{\Span}{span}								
\DeclareMathOperator{\Hess}{Hess}								
\DeclareMathOperator{\Proj}{P}                                  
\DeclareMathOperator{\Deriv}{D}                                  
\DeclareMathOperator{\id}{id}                                  
\DeclareMathOperator{\dist}{dist}								
\DeclareMathOperator{\Exp}{Exp}								
\newcommand{\Mf}{{\mathcal{M}}}       							
\newcommand{\Mfr}{{\mathcal{M}_\mathbf{r}}}       							
\newcommand{\T}{{\mathrm{T}}}       							
\newcommand{\R}{{\mathbb{R}}}       							

\def\volumeyear{2017}

\begin{document}

\runningheads{Gennadij~Heidel~and~Volker~Schulz}{A Riemannian trust-region method for low-rank tensor completion}

\title{A Riemannian trust-region method for low-rank tensor completion}

\author{Gennadij~Heidel\corrauth~and~Volker Schulz}

\address{Fachbereich IV - Mathematik, Universit\"at Trier, 54286 Trier, Germany}

\corraddr{heidel@uni-trier.de}

\begin{abstract}
The goal of tensor completion is to fill in missing entries of a partially known tensor (possibly including some noise) under a low-rank constraint. This may be formulated as a least-squares problem. The set of tensors of a given multilinear rank is known to admit a Riemannian manifold structure, thus methods of Riemannian optimization are applicable.

In our work, we derive the Riemannian Hessian of an objective function on the low-rank tensor manifolds using the
Weingarten map, a concept from differential geometry. We discuss the convergence properties of Riemannian
trust-region methods based on the exact Hessian and standard approximations, both theoretically and numerically. We compare our approach to Riemannian tensor completion methods from recent literature, both in terms of convergence behaviour and computational complexity. Our examples include the completion of randomly generated data with and without noise and recovery of
multilinear data from survey statistics.
\end{abstract}

\keywords{Riemannian optimization; multilinear rank; low-rank tensors; Tucker decomposition; Riemannian Hessian; trust-region methods}

\maketitle

\vspace{-6pt}

\section{Introduction} \label{sec:introduction}

In this paper, we discuss optimization techniques on the manifold of tensors of a given rank. We consider least-squares problems of the 
form
\begin{equation} \label{eqn:problem}
	\begin{split}
		&\min_\mathbf{X}\, f(\mathbf{X}) = \frac{1}{2} \normb{\Proj_\varOmega\mathbf{X}-\Proj_\varOmega\mathbf{A}}^2\\
		\text{s.\,t.} ~ &\mathbf{X} \in \Mfr \coloneqq \big\{ \mathbf{X}
		\in \R^{n_1\times\dotsb\times n_d} \, \big| \, \rank(\mathbf{X}) = \mathbf{r}\big\},
	\end{split}
\end{equation}
where $\rank(\mathbf{X}) \in \R^d$ denotes the multilinear rank of a tensor $\mathbf{X}$, and
$\Proj_\varOmega: \R^{n_1\times\dotsb\times n_d}\rightarrow \R^{n_1\times\dotsb\times n_d}$ is a linear operator. A typical choice found in the literature is
\begin{equation*}
	[\Proj_\varOmega\mathbf{X}]_{i_1\dotsc i_d} \coloneqq \begin{cases}x_{i_1\dotsc i_d}&\text{if~} (i_1,\dotsc,i_d) \in \varOmega,\\
	0 &\text{otherwise},\end{cases}
\end{equation*}
where $\varOmega \subset \{1,\dotsc,n_1\}\times\dotsb\times\{1,\dotsc,n_d\}$ denotes the sampling set, i.\,e. we assume
that ${\mathbf{A}\in\R^{n_1\times\dotsb\times n_d}}$ is a tensor whose entries with indices in $\varOmega$ are known.

The tensor completion problem is a generalization of the matrix completion problem, see the page by Ma et al. \cite{ma2017low} for an overview of methods and applications
in the context of convex optimization. Early work on tensor completion has been done by Liu et al. \cite{liu2013tensor}, who consider the problem
\begin{equation} \label{eqn:problem_liu}
		\min_\mathbf{X}\, \norm{\mathbf{X}}_* \text{\;\;\;s.\,t.\;\,} \Proj_\varOmega\mathbf{X}=\Proj_\varOmega\mathbf{A}
\end{equation}
in the context of image data recovery, where $\norm{\,\cdot\,}_*$ is a generalized nuclear norm. Note that \eqref{eqn:problem_liu} can be viewed as the
dual of \eqref{eqn:problem}. It ensures convexity for the tensor completion problem at the cost of losing the underlying manifold structure of low-rank
tensors. Specifically, it does not give a low-rank solution in the presence of noise, i.\,e. if $\mathbf{A} \notin \Mfr$;
in this case, an additional routine may be needed to truncate the result to low rank. Signoretto et al. \cite{signoretto2010nuclear} and Gandy et al. \cite{gandy2011tensor} choose a Tikhonov-like approach by minimizing the
a penalized unconstrained function
\begin{equation*}
		\min_\mathbf{X}\, \frac{1}{2} \normb{\Proj_\varOmega\mathbf{X}-\Proj_\varOmega\mathbf{A}}^2 + \frac{\mu}{2} \norm{\mathbf{X}}_*.
\end{equation*}

A Riemannian CG method for \eqref{eqn:problem} has been proposed by Kressner et al. \cite{kressner2014low}, which is an extension of Vandereycken's
earlier work \cite{vandereycken2013low} for the matrix completion problem. The authors show rapid linear convergence of their method with satisfactory
reconstruction of missing data for a range of applications. Other Riemannian approaches for matrix completion include the work by Ngo/Saad
\cite{ngo2012scaled} and Mishra et al. \cite{mishra2014fixed}, who use a product Graßmann quotient manifold structure.

In recent research, second-order methods in Riemannian optimization have generated considerable interest in order to find superlinearly converging
methods, see the overview by Absil et al. \cite[Chapters~6--8]{absil2008optimization} and the references therein. Boumal/Absil \cite{boumal2011rtrmc} apply 
these techniques to matrix completion in the Graßmannian framework. Vandereycken \cite[Subsection~2.3]{vandereycken2013low} derives the Hessian
for Riemannian matrix completion with an explicit expression of the singular values. In the higher-order tensor case, Eldén/Savas \cite{elden2009newton}
propose a Newton method for computing a rank-$\mathbf{r}$ tensor approximation, using a Graßmannian approach. Ishteva et al. \cite{ishteva2011best} extend these ideas to construct a Riemannian trust-region scheme.

In this paper, we propose a Riemannian trust-region scheme for \eqref{eqn:problem} using explicit Tucker decompositions and 
compare it to a state-of-the-art Riemannian CG as used in \cite{kressner2014low}. We derive the exact expression of the Riemannian Hessian on $\Mfr$ for this manifold geometry by 
using the Weingarten map proposed by Absil et al. \cite{absil2013extrinsic}. Our work focuses on the application case of tensor completion and
contains tensor approximation as the special case of full sampling, i.\,e. $| \varOmega | = \prod_i n_i$.

The rest of the paper is organized as follows. In Section~\ref{sec:tensors}, we cite some basic results about tensor 
arithmetic and the manifold of low-rank tensors. In Section~\ref{sec:geometry}, we present a brief overview of Riemannian
optimization and prove our main result, the Riemannian Hessian on $\Mfr$. In Section~\ref{sec:models}, we explain
the Riemannian trust-region methods based on exact and approximate Hessian evaluations. In Section~\ref{sec:experiments},
we present the some numerical experiments for our method on synthetic data and a standard test data set from multilinear statistics.

\section{Low-rank tensors} \label{sec:tensors}

In this section, we collect some basic concepts and results on the Tucker decomposition and multilinear rank of tensors 
needed for our work. First, we define notations and results of general tensor arithmetic, as laid out in the survey paper 
\cite{kolda2009tensor}. Then, we introduce the manifold geometry of $\Mfr$, see \cite{uschmajew2013geometry,koch2010dynamical,kressner2014low}.

\subsection{Multilinear rank and Tucker decomposition} \label{subsec:tucker}

For a tensor $\mathbf{A}\in \R^{n_1\times\dotsb\times n_d}$, the matrix
\begin{equation*}
	A_{(i)} \in \R^{n_i\times\prod_{j\neq i}n_j},
\end{equation*}
such that the row index of $A_{(i)}$ is the $i$th modes of $\mathbf{A}$ and the column index is a multi-index of the remaining $d-1$ 
modes, in lexicographic order, is called the \textit{mode-$i$ matricization} of $\mathbf{A}$. It may be viewed as a $d$-order generalization 
of the matrix transpose, since, for $d=2$, it holds that $A_{(1)}=A$ and $A_{(2)}=A^\T$. We denote the re-tensorization of 
a matricized tensor by a superscript index, i.\,e. $(A_{(i)})^{(i)}=\mathbf{A}$.

The \textit{multilinear rank} of a tensor $\mathbf{A}$ is the $d$-tuple
\begin{equation*}
	\rank(\mathbf{A}) = \big( \rank(A_{(1)}), \dotsc, \rank(A_{(d)}) \big),
\end{equation*}
with $\rank(\,\cdot\,)$ on the right-hand side of the equation denoting the matrix rank. In contrast to the matrix case, the ranks of
different matricizations of a tensor may be different, e.\,g. consider $\mathbf{A}\in\R^{2\times 2\times 2}$, given by its mode-$1$ matricization
\begin{equation*}
	A_{(1)}= \begin{bmatrix}
	1&0&0&0\\
	0&1&0&0
	\end{bmatrix}.
\end{equation*}
Then, the other matricizations are
\begin{equation*}
	A_{(2)} = A_{(1)}, ~ A_{(3)} = \begin{bmatrix}
	1&0&0&1\\
	0&0&0&0
	\end{bmatrix},
\end{equation*}
so clearly $\rank(\mathbf{X}) = (2,2,1)$.

The \textit{$i$-mode product} of $\mathbf{A}$ with a matrix $M \in \R^{m\times n_i}$ is defined as
\begin{equation*}
	\mathbf{B} = \mathbf{A} \times_i M \iff B_{(i)} = MA_{(i)}, ~ \mathbf{B} \in
	\R^{n_1\times\dotsb\times n_{i-1}\times m \times n_{i+1}\dotsb\times n_d}.
\end{equation*}
It is worth noting that, for different modes, the order of multiplications is irrelevant, i.\,e.
\begin{equation} \label{eqn:diff_modes}
	\mathbf{A} \times_i M \times_j N = \mathbf{A} \times_j N \times_i M \text{\;\;\;if\;\,} i \neq j.
\end{equation}
If the modes are equal, then
\begin{equation} \label{eqn:same_mode}
	\mathbf{A} \times_i M \times_i N = \mathbf{A} \times_i (NM).
\end{equation}

A Frobenius inner product on $\R^{n_1\times\dotsb\times n_d}$ is given by
\begin{equation*}
	\dotprod{\mathbf{A}}{\mathbf{B}} \coloneqq \tr\big(A_{(1)}^\T B_{(1)}\big) =
	\dotsb = \tr\big(A_{(d)}^\T B_{(d)}\big) = 
	\sum_{i_1=1}^{n_1}\dotsi\sum_{i_d=1}^{n_d}a_{i_1\dotso i_d}b_{i_1\dotso i_d},
\end{equation*}
with the induced norm $\norm{\mathbf{A}} \coloneqq \sqrt{\dotprod{\mathbf{A}}{\mathbf{A}}}$.

A tensor $\mathbf{X}$ with $\rank(\mathbf{X})=\mathbf{r}=(r_1,\dotsc,r_d)$ can be represented in the \textit{Tucker decomposition} \cite{tucker1966some}
\begin{equation} \label{eqn:tuckerdecomp}
	 \mathbf{X} = \mathbf{C} \times_1 U_1 \dotsb \times_d U_d = \mathbf{C} \bigtimes_{i=1}^d U_i,
\end{equation}
with a \textit{core tensor} $\mathbf{C} \in \R^{r_1\times\dotsb\times r_d}$ with $\rank(\mathbf{C}) = \mathbf{r}$ and \textit{basis 
matrices} $U_i \in \R^{n_i\times r_i}$ with linearly independent columns. Without loss of generality, it can be assumed that the
basis matrices have orthonormal columns, i.\,e. $U_i^\mathrm{T}U_i=I$. If for some $i$ this is not the case, a QR factorization
$U_i = \widetilde{U}_i R$, with $\widetilde{U}_i$ orthonormal and $R$ regular and $\widetilde{\mathbf{C}}=(RC_{(i)})^{(i)}$ gives
the required property.

A rank-$\mathbf{r}$ approximation to a tensor $\mathbf{A}$ can be computed by the \textit{truncated higher-order SVD (HOSVD)} 
\cite{lathauwer2000multilinear}: Let $\Proj_{r_i}^i$ be a the best rank-$r_i$ approximation operator in the $i$th mode, i.\,e.
$\Proj_{r_i}^i \mathbf{A}  = (U_i U_i^\mathrm{T} A_{(i)})^{(i)}$, where $U_i$ denotes the matrix of the $r_i$ dominant left
singular vectors of $A_{(i)}$. Then the rank-$\mathbf{r}$ truncated HOSVD operator
$\Proj_\mathbf{r}^\mathrm{HO}$ is given by
\begin{equation} \label{eqn:hosvd}
	\Proj_\mathbf{r}^\mathrm{HO} \mathbf{A} \coloneqq \Proj_{r_1}^1 \dotsm \Proj_{r_d}^d \mathbf{A}.
\end{equation}

In contrast to the matrix case, the HOSVD does not yield a best rank-$\mathbf{r}$ approximation, but only a quasi-best-approximation
\cite[Property 10]{lathauwer2000multilinear} with a constant which deteriorates with respect to the number of modes:
\begin{equation} \label{eqn:quasi_best}
	\normb{\mathbf{A}-\Proj_\mathbf{r}^\mathrm{HO} \mathbf{A}} \leq
	\sqrt{d} \min_{\mathbf{X}\in\Mf_\mathbf{r}} \norm{\mathbf{A}-\mathbf{X}}.
\end{equation}

\subsection{Riemannian manifold structure} \label{subsec:manifold}

In \cite{uschmajew2013geometry}, the authors show that the set $\Mf_\mathbf{r}$ of tensors of fixed multilinear rank
$\mathbf{r}=(r_1,\dotsc,r_d)$ forms a smooth embedded submanifold of $\R^{n_1\times\dotsb\times n_d}$. By counting the degrees of 
freedom in \eqref{eqn:tuckerdecomp}, it follows that 
\begin{equation*}
	\dim(\Mfr) = \prod_{i=1}^d r_i + \sum_{i=1}^d r_i n_i - r_i^2,
\end{equation*}
where the last term accounts for the fact that the Tucker decomposition is invariant to simultaneous transformation of the basis matrix
with an invertible matrix and the core tensor with its inverse; as described in the previous subsection. Being a submanifold 
of the Euclidean space $(\R^{n_1\times\dotsb\times n_d},\dotprod{\,\cdot\,}{\,\cdot\,})$, the manifold
$\Mfr$ can be endowed with a Riemannian structure in a natural way with the Frobenius inner product
$\dotprod{\,\cdot\,}{\,\cdot\,}$ as the Riemannian metric. 

As is proven in \cite[Subsection 2.3]{koch2010dynamical}, the tangent space of $\Mfr$ at
$\mathbf{X}=\mathbf{C}\bigtimes_{i=1}^d U_i$ is parametrized as
\begin{equation}
	T_\mathbf{X}\Mfr = \bigg\{ \dot{\mathbf{C}}\bigtimes_{i=1}^d U_i +
	\sum_{i=1}^d \mathbf{C} \times_i \dot{U}_i \bigtimes_{j \neq i}U_j \,\bigg|\,
	\dot{\mathbf{C}} \in \R^{r_1\times\dotsb\times r_d},
	~ \dot{U}_i \in \R^{n_i\times r_i} \text{~with~} \dot{U}_i^\mathrm{T} U_i = O \bigg\},
\end{equation}
and the orthogonal projection
$\Proj_\mathbf{X}:\R^{n_1\times\dotsb\times n_d} \rightarrow T_\mathbf{X}\Mfr$ is given by
\begin{equation} \label{eqn:orth_proj}
	\mathbf{A} \mapsto \bigg(\mathbf{A}\bigtimes_{j=1}^d U_j^\mathrm{T}\bigg)\bigtimes_{i=1}^d U_i
	+ \sum_{i=1}^d \mathbf{A} \times_i
	\Bigg( \Proj_{U_i}^\perp
	\bigg[ \mathbf{A} \bigtimes_{j \neq i}U_j^\mathrm{T} \bigg]_{(i)} C_{(i)}^+ \Bigg) \bigtimes_{k \neq i}U_k,
\end{equation}
where $C_{(i)}^+$ denotes the Moore-Penrose pseudoinverse of $C_{(i)}$. Note that $C_{(i)}$ has full row rank, i.\,e. 
$C_{(i)}^+=C_{(i)}^\T(C_{(i)}C_{(i)}^\T)^{-1}$. We use $\Proj_{U_i}^\perp = I_{n_i}-U_i U_i^\T$ to denote the orthogonal projection
onto $\Span(U_i)^\perp$.

Furthermore, it can be shown that the HOSVD \eqref{eqn:hosvd} is locally a $C^\infty$ function in the manifold topology of $\Mfr$, see
\cite[Proposition 2.1]{kressner2014low} for further details. This allows us its use in continuous optimization, as we will se in the next section.

\section{The geometry of $\Mfr$ and Riemannian optimization}  \label{sec:geometry}

To construct optimization methods on $\Mfr$, we collect some basic concepts from the theory of optimization on
manifolds. Our exposition follows the overview book \cite{absil2008optimization}. Furthermore, we need to define
and calculate the first and second derivatives of functions on $\Mfr$. In Corollary~\ref{cor:curvature_term}, we prove
our main result, an explicit expression for the Riemannian Hessian on $\Mfr$. In the following, we will denote a
Riemannian manifold by $\Mf$ and its elements by $x,y,\dotsc\in\Mf$, when citing general results, and the manifold of 
tensors of fixed multilinear rank by $\Mfr$ and its elements by $\mathbf{X},\mathbf{Y},\dotsc\in\Mfr$.

\subsection{Retraction and vector transport} \label{subsec:retr_transp}

Since a manifold is in general not a linear space, the calculations required for a continuous optimization method need to be performed in a tangent
space. Therefore, in each step, the need arises to map points from a tangent space to the manifold in order to generate the new
iterate. The theoretically superior choice of such a mapping is the \textit{exponential map}, which moves a point $x$ on the 
manifold along the geodesic locally defined by a vector in the tangent space $T_x\Mf$. However, computing the 
exponential map is prohibitively expensive in most situations, and it is shown in \cite{absil2008optimization} that a first-order
approximation, as specified in the following definition, is sufficient for many convergence results. 
\begin{definition}
	A \textit{retraction} on a manifold $\Mf$ is a smooth mapping $R$ from the tangent bundle $T\Mf$ onto $\Mf$
	with the following properties. Let $R_x$ denote the restriction of $R$ to $T_x\Mf$.
	\begin{enumerate}[label=(\roman*)]
		\item $R_x(0_x)=x$, where $0_x$ denotes the zero element of $T_x\Mf$.
		\item With the canonical identification $T_{0_x}T_x\Mf \simeq T_x\Mf$, the mapping $R_x$ satisfies
		the \textit{rigidity condition}
		\begin{equation*}
			\Deriv R_x(0_x) = \mathrm{id}_{T_x\Mf},
		\end{equation*}
		where $\mathrm{id}_{T_x\Mf}$ denotes the identity mapping on $T_x\Mf$.
	\end{enumerate}
\end{definition}

Furthermore, ``comparing'' tangent vectors at distinct
points on the manifold will be useful. The following definition gives us a way to transport a tangent vector $\xi \in T_x\Mf$ to the tangent
space $T_{R_x(\eta)}\Mf$ for some $\eta\in T_x\Mf$ and some retraction $R$.
\begin{definition}
	A \textit{vector transport} on a manifold $\Mf$ is a smooth mapping
	\begin{equation*}
		T\Mf \oplus T\Mf \rightarrow T\Mf: (\eta,\xi) \mapsto \mathcal{T}_{\eta}(\xi),
	\end{equation*}
	satisfying the following properties for all $x \in \Mf$:
	\begin{enumerate}[label=(\roman*)]
		\item (Associated retraction) There exists a retraction $R$, called the \textit{retraction associated with}
		$\mathcal{T}$, such that, for all $\eta, \xi$, it holds that $\mathcal{T}_{\eta}\xi \in T_{R_x(\eta)}\Mf$.
		\item (Consistency) $\mathcal{T}_{0_x}\xi = \xi$ for all $\xi \in T_x\Mf$.
		\item (Linearity) The mapping $\mathcal{T}_{\eta}:T_x\Mf \rightarrow T_{R_x(\eta)}\Mf,~
		\xi \mapsto\mathcal{T}_{\eta}\xi$ is linear.
	\end{enumerate}
\end{definition}

For $\Mf=\Mfr$, a retraction is given by the HOSVD, i.\,e.
$R_\mathbf{X}(\xi)=\Proj_\mathbf{r}^\mathrm{HO} (\mathbf{X}+\xi)$. This is a consequence of the smoothness of the HOSVD (cf. Subsection~\ref{subsec:manifold}) and the quasi-best approximation property \eqref{eqn:quasi_best}. Details may be found in \cite[Proposition~3]{kressner2014low}. A vector transport associated with a retraction $R$ is given by the orthogonal projection onto the tangent space, i.\,e.
$\mathcal{T}_\eta(\xi) = \Proj_{R_\mathbf{X}(\eta)}(\xi)$, see \cite[Subsection~8.1.3]{absil2008optimization};
in our case, this is the formula \eqref{eqn:orth_proj}. The efficient
implementation of these operations is discussed in \cite[Subsections~3.3--3.4]{kressner2014low}. A geometrical 
interpretation is shown in Figure~\ref{fig:retr_transp}.

\begin{figure}
	\begin{tabular}{ccc}
		\hspace{10mm}&
		\begin{tikzpicture}[scale=0.35]
			\draw [line width=0.5mm, black ] (0,0) to [bend left=45] (7.8,6.3)
			to [bend left=45] (14,-1.5);
			\draw [line width=0.5mm, black ] (1.035,3) to [bend left=55] (7.1,1.7)
			to [bend left=12] (8.2,-2.6);
			\draw [line width=0.5mm, black ] (8.2,-2.57) to [bend left=25] (14,-1.47);
			\draw [line width=0.5mm, black ] (0,0) to [bend left=12] (6.1,3);
			\draw [->,line width=0.5mm](7,5) -- ++(4,-1);
			\draw [line width=0.3mm] (4,4) -- ++(10,-2.5);
			\draw [line width=0.3mm] (4,4) -- ++(2.5,4);
			\draw [line width=0.3mm] (6.5,8) -- ++(10,-2.5);
			\draw [line width=0.3mm] (14,1.5) -- ++(2.5,4);
			\draw [dotted,line width=0.5mm] (7,5) to [bend left=25] node{.}(11,1);
			\filldraw (7,5) circle (4pt);
			\filldraw (11,1) circle (4pt);
			\node at (6,5.5) {$\mathbf{X}$};
			\node at (12.5,0) {$R_\mathbf{X}(\xi)$};
			\node at (11.5,3.5) {$\xi$};
			\node at (6.5,-1) {$\Mfr$};
			\node at (3.5,7.5) {$T_{\mathbf{X}} \Mfr$};
			\draw[opacity=0] (0,0) --(20,0);
		\end{tikzpicture}&
		\begin{tikzpicture}[scale=0.35]
			\draw [line width=0.5mm, black ] (0,0) to [bend left=45] (7.8,6.3)
			to [bend left=45] (14,-1.5);
			\draw [line width=0.5mm, black ] (1.035,3) to [bend left=55] (7.1,1.7)
			to [bend left=12] (8.2,-2.6);
			\draw [line width=0.5mm, black ] (8.2,-2.57) to [bend left=25] (14,-1.47);
			\draw [line width=0.5mm, black ] (0,0) to [bend left=12] (6.1,3);
			\draw [->,line width=0.5mm](7,5) -- ++(4,-1);
			\draw [line width=0.3mm] (4,4) -- ++(10,-2.5);
			\draw [line width=0.3mm] (4,4) -- ++(2.5,4);
			\draw [line width=0.3mm] (6.5,8) -- ++(10,-2.5);
			\draw [line width=0.3mm] (14,1.5) -- ++(2.5,4);
			\draw [dotted,line width=0.5mm] (7,5) to [bend left=25] node{.}(11,1);
			\filldraw (7,5) circle (4pt);
			\filldraw (11,1) circle (4pt);
			\node at (6,5.5) {$\mathbf{X}$};
			\node at (12.5,0) {$R_\mathbf{X}(\eta)$};
			\node at (11.5,3.5) {$\eta$};
			\node at (6.5,-1) {$\Mfr$};
			\node at (3.5,7.5) {$T_{\mathbf{X}} \Mfr$};
			\draw [->,line width=0.5mm](7,5) -- ++(1,-1.5);
			\node at (6.5,4) {$\xi$};
			\draw [->,line width=0.5mm](11,1) -- ++(-1,-2);
			\node at (9,0.5) {$\mathcal{T}_\eta(\xi) $};
		\end{tikzpicture}
	\end{tabular}\\
	\caption{Retraction (left) and vector transport (right) on $\Mfr$.} \label{fig:retr_transp}
\end{figure}
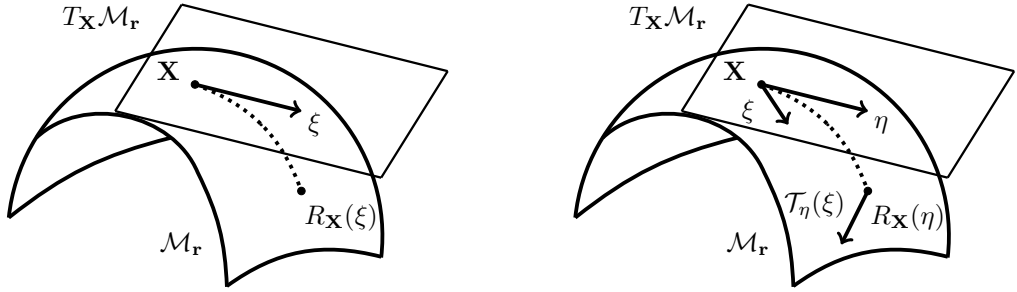

\subsection{The Riemannian gradient} \label{subsec:grad}

The low-rank Tucker manifold $\Mfr$ being a submanifold a Euclidean space, the gradient of a real-valued function
defined on it can be easily calculated by projecting the Euclidean gradient onto the tangent space.
\begin{lemma} \label{lem:riemann_grad}
	\cite[Section 3.6.1]{absil2008optimization} Let $\Mf$ be a Riemannian submanifold of a Euclidean space $E$. Let
	$\bar{f}:E \rightarrow \R$ be a function with Euclidean gradient $\grad \bar{f}(x)$ at point $x \in \Mf$. Then the
	Riemannian gradient of $f \coloneqq \bar{f}|_\Mf$ is given by $\grad f(x) = \Proj_x \grad \bar{f} (x)$, where $
	\Proj_x$ denotes the orthogonal projection onto the tangent space $T_x\Mf$.
\end{lemma}

Then, by Lemma \ref{lem:riemann_grad}, the Riemannian gradient of the tensor completion cost function is given by
\begin{equation}
	\grad f(\mathbf{X}) = \Proj_\mathbf{X}(\Proj_\varOmega\mathbf{X}-\Proj_\varOmega\mathbf{A}).
\end{equation}
Using the sparsity of $\Proj_\varOmega\mathbf{X}-\Proj_\varOmega\mathbf{A}$, a gradient evaluation requires $\mathcal{O}(r^d(|\varOmega|+n)+r^{d+1})$ operations,
cf. \cite[Subsection~3.1]{kressner2014low}, where we assume that the $r_i$ and $n_i$ are constant in each mode for simplicity of notation.

\subsection{The Riemannian Hessian} \label{subsec:hess}

By definition, the \textit{Riemannian Hessian} of a real-valued function $f$ on a Riemannian manifold $\Mf$ is a linear mapping
\begin{equation}
	\Hess f(x) [\xi] = \nabla_{\xi}\grad f(x),
\end{equation}
where $\nabla$ denotes the \textit{Riemannian connection} on $\Mf$, cf. \cite[Definition 5.5.1]{absil2008optimization}. A
finite-difference approximation can be defined in different ways. An intuitive formula is given by
\begin{equation} \label{eqn:fd}
	H^{\mathrm{FD}} [\xi] = \frac{\mathcal{T}_\xi\grad f(R_x(h\xi))-\grad f(x)}{h},
\end{equation}
see, for example, \cite[Subsection~8.2.1]{absil2008optimization}. However, such a mapping will in general not be linear
\cite{boumal2015riemannian}, and should be applied with care, as theoretical understanding is yet incomplete. 

On a Riemannian submanifold of a Euclidean space, the Riemannian connection is just the orthogonal projection of the directional 
derivative, i.\,e.
\begin{equation}
		\Hess f(x) [\xi_x] = \Proj_x\big(\Deriv \grad f(x)[\xi_x]\big),
\end{equation}
and using Lemma \ref{lem:riemann_grad}, we get the following result.

\begin{lemma} \label{lem:riemann_hess}
	\cite[Section 5.3.3]{absil2008optimization} Let $\Mf$ be a Riemannian submanifold of a Euclidean space $E$. Let
	$\bar{f}:E \rightarrow \R$ be a function with Euclidean gradient $\grad \bar{f}(x)$ at point $x \in \Mf$. Then the
	Riemannian Hessian of $f \coloneqq \bar{f}|_\Mf$ is given by
	\begin{equation} \label{eqn:lem_riemann_hess}
		\Hess f(x) [\xi_x] = \Proj_x \Deriv\big( \Proj_x \grad \bar{f} (x)\big).
	\end{equation}
\end{lemma}

Using the chain rule, we can write \eqref{eqn:lem_riemann_hess} as
\begin{align}
	\Hess f(x) [\xi] &= \Proj_x \Deriv\big( \Proj_x \grad \bar{f} (x)\big)\nonumber\\
	&= \Proj_x \Hess\bar{f}(x)[\xi_x] + \Proj_x  \Deriv_\xi \Proj_x \grad \bar{f} (x), \label{eqn:curvature_term}
\end{align}
where we view $x\mapsto\Proj_x$ as an operator-valued function and denote its directional derivative by $\Deriv_\xi$. We observe that he 
first term in \eqref{eqn:curvature_term} is just the orthogonal projection of the Euclidean Hessian, while the 
second one depends on the curvature of the manifold $\Mf$. Indeed, the second term is equal to zero when $\Mf$ is flat, i.\,e. a linear
subspace of the embedding Euclidean space, cf. \cite[Subsection~4.1]{kressner2016preconditioned}. Clearly, the main challenge in calculating the Riemannian Hessian in
\eqref{eqn:curvature_term} is the derivative of the projection operator. In \cite[Section 3]{absil2013extrinsic}, the authors
show the following result using the \textit{Weingarten map}.

\begin{lemma}
	Let $\Mf$ be a Riemannian submanifold of a Euclidean space $\mathcal{E}$. For any $x\in\Mf$, let $\Proj_x$ denote the orthogonal
	projection onto the tangent space $T_x\Mf$, and $\Proj_x^\perp \coloneqq \id_\mathcal{E} - \Proj_x$ the orthogonal projection
	on its orthogonal complement $(T_x\Mf)^\perp$. We view $x\mapsto\Proj_x$ as an operator-valued function and denote its Gâteaux
	derivative at point $x$ in the direction of $\xi\in T_x\Mf$ by $\Deriv_\xi \Proj_x$. Then
	\begin{equation}
		\Proj_x\Deriv_\xi \Proj_x u = \Proj_x \Deriv_\xi \Proj_x \big(\Proj_x^\perp u\big),
	\end{equation}
	for all $x \in \Mf$, $\xi \in T_x\Mf$ and $u \in \mathcal{E}$.
\end{lemma}

This result can be applied to the case of the low-rank Tucker manifold $\Mf = \Mfr$. First, we calculate the derivative
$\Deriv_\xi \Proj_\mathbf{X}$.

\begin{lemma} \label{lem:curvature_term}
	Let $\mathbf{X} \in \Mfr$ be a tensor on the low-rank manifold, given by the factorization
	$\mathbf{X} = \mathbf{C} \bigtimes_{i=1}^d U_i$, and let $\xi \in T_\mathbf{X}\Mf_{\mathbf{r}}$, given by the variations
	\begin{equation*}
		\xi = \dot{\mathbf{C}}\bigtimes_{i=1}^d U_i + \sum_{i=1}^d \mathbf{C} \times_i \dot{U}_i \bigtimes_{j\neq i} U_j.
	\end{equation*}
	We use the notations $\Proj_{U_i} = U_iU_i^\T$, $\Proj_{U_i}^\perp = I_{n_i}-U_iU_i^\T$ and
	$\dot{\Proj}_{U_i}=\dot{U}_iU_i^\T+U_i\dot{U}_i^\T$.
	Then, for any $\mathbf{E}\in\R^{n_1\times\dotsb\times n_d}$, the derivative of $\Proj_\mathbf{X}$ in the direction of $\xi$
	is given by
	\begin{align*}
		\Deriv_\xi \Proj_\mathbf{X} \mathbf{E} = \sum_{i=1}^d \Bigg\{
		&\mathbf{E} \times_i \dot{\Proj}_{U_i} \bigtimes_{j\neq i} \Proj_{U_j}\\
		+\,&\dot{\mathbf{C}}\times_i \Bigg( \Proj_{U_i}^\perp \bigg[ \mathbf{E}\bigtimes_{j\neq i} U_j^\T \bigg]_{(i)} C_{(i)} \Bigg)
		\bigtimes_{k\neq i} U_k\\
		-\,&\mathbf{C}\times_i \Bigg( \dot{\Proj}_{U_i} \bigg[ \mathbf{E}\bigtimes_{j\neq i} U_j^\T \bigg]_{(i)} C_{(i)} \Bigg)
		\bigtimes_{k\neq i} U_k\\
		+\,&\sum_{l\neq i}\mathbf{C}\times_i \Bigg( \Proj_{U_i}^\perp \bigg[ \mathbf{E} \times_l \dot{U}_l^\T
		\bigtimes_{l\neq j\neq i} U_j^\T \bigg]_{(i)} C_{(i)} \Bigg) \bigtimes_{k\neq i} U_k\\
		+\,&\mathbf{C}\times_i \Bigg( \Proj_{U_i}^\perp \bigg[ \mathbf{E}\bigtimes_{j\neq i} U_j^\T \bigg]_{(i)}
		\bigg[ \Big( I - C_{(i)}^+C_{(i) }\Big)
		\dot{C}_{(i)}^\T C_{(i)}^{+\T}C_{(i)}^+ - C_{(i)}^+\dot{C}_{(i)}C_{(i)}^+ \bigg] \Bigg)
		\bigtimes_{k\neq i} U_k\\
		+\,&\sum_{l\neq i}\mathbf{C}\times_i \Bigg( \Proj_{U_i}^\perp \bigg[ \mathbf{E}\bigtimes_{j\neq i} U_j^\T \bigg]_{(i)}
		C_{(i)} \Bigg) \times_l \dot{U}_l \bigtimes_{l\neq k\neq i} U_j
		\Bigg\},
	\end{align*}
	where $I = I_{\prod_{j\neq i}r_j}$ is the identity matrix of the appropriate size.
\end{lemma}
\begin{proof}
	The formula can be obtained by identifying the tensor $\mathbf{X}$ with the factors in the Tucker decomposition and viewing the 
	orthogonal projection defined in \eqref{eqn:orth_proj} as a function
	\begin{equation*}
		\Proj_{\,\cdot\,}\mathbf{E}: \R^{r_1 \times \dotsb \times r_d}\times  \R^{n_1\times r_1} \times \dotsb \times \R^{n_d\times r_d}
		\rightarrow \R^{n_1 \times \dotsb \times n_d}, ~ (\mathbf{C},U_1,\dotsc,U_d) \mapsto \Proj_\mathbf{X} \mathbf{E},
	\end{equation*}
	for any $\mathbf{E}\in\R^{n_1 \times \dotsb \times n_d}$.
	For calculating the derivative of the pseudoinverse, we use the formula
	given in \cite[Theorem~4.3]{golub1973differentiation}, i.\,e.
	\begin{equation*}
		\Deriv_{\dot{C}}\big( C^+\big) = \big( I - C^+C\big)\dot{C}^\T C^{+\T}C^+ 
		+ C^+C^{+\T}\dot{C}^\T \big(  CC^+ - I \big)
		- C^+\dot{C}C^+,
	\end{equation*}
	and note that, here, the second term vanishes since $C = C_{(i)}$ has full row rank, and thus the pseudoinverse is a right inverse.
\end{proof}

Using this result, we can immediately evaluate the curvature term in \eqref{eqn:curvature_term}.

\begin{corollary} \label{cor:curvature_term}
	We use the setting of Lemma~\ref{lem:curvature_term} and denote the orthogonal projection onto
	$(T_\mathbf{X}\Mf_{\mathbf{r}})^\perp$ by $\Proj_\mathbf{X}^\perp\coloneqq\mathrm{id}-\Proj_\mathbf{X}$. Then
	\begin{equation*}
		\Proj_\mathbf{X} \Deriv_\xi\Proj_\mathbf{X} \Proj_\mathbf{X}^\perp \mathbf{E} =
		\widetilde{\mathbf{C}}\bigtimes_{i=1}^d U_i + \sum_{i=1}^d \mathbf{C} \times_i \widetilde{U}_i \bigtimes_{j\neq i} U_j
		\in T_\mathbf{X}\Mfr,
	\end{equation*}
	with
	\begin{align*}
		\widetilde{\mathbf{C}} &= \sum_{j=1}^d \bigg( \mathbf{E} \times_j \dot{U}_j^\T \bigtimes_{k\neq j}U_k^\T
		- \mathbf{C} \times_j \Big( \dot{U}_j^\T \big[\mathbf{E}\bigtimes_{k\neq j} U_j^\T\big]_{(j)} C_{(j)}^+ \Big) \bigg),\\
		\widetilde{U}_i &= \Proj_{U_i}^\perp \bigg( 
		\big[\mathbf{E}\bigtimes_{j\neq i} U_j^\T\big]_{(i)} \big(I - C_{(i)}^+C_{(i)}\big) \dot{C}_{(i)}^\T C_{(i)}^{+\T} 
		+ \sum_{k\neq i} \big[\mathbf{E} \times_k \dot{U}_k^\T \bigtimes_{k\neq j \neq i} U_j^\T \big]_{(i)} \bigg) C_{(i)}^+,
	\end{align*}
\end{corollary}
\begin{proof}
	The result follows by applying Lemma~\ref{lem:curvature_term} to $\Proj_\mathbf{X}^\perp \mathbf{E}\in\R^{n_1\times\dotsb\times n_d}$ 
	after some lengthy but 
	straightforward calculations, using the orthonormality relations $\dot{U}_i^\T U_i = O$, $U_i^\T U_i = I$ and the rules
	\eqref{eqn:diff_modes} and \eqref{eqn:same_mode} for the matrix-tensor product.
\end{proof}

Thus, the Riemannian Hessian of the function $f:\Mfr \rightarrow \R$,
\begin{equation*}
	f(\mathbf{X}) = \frac{1}{2} \normb{\Proj_{\varOmega}\mathbf{X}-\Proj_{\varOmega}\mathbf{A}}^2,
\end{equation*}
can be written as
\begin{equation} \label{eqn:hessian}
	\Hess f(\mathbf{X}) [\xi] = \Proj_{\varOmega} (\xi) +
	\Proj_\mathbf{X} \Deriv_\xi\Proj_\mathbf{X} \Proj_\mathbf{X}^\perp(\Proj_\varOmega \mathbf{X} - \Proj_\varOmega \mathbf{A})
\end{equation}
and the second term can be evaluated with Corollary~\ref{cor:curvature_term}.

Note that for an efficient computation of the terms $\widetilde{U}_i$, it is advantageous to multiply out the term containing $I - C_{(i)}^+C_{(i)}$. Then, the computation of $\Hess f(\mathbf{X}) [\xi]$ for any given
$\xi\in T_\mathbf{X}\Mfr$ has the same complexity as the computation of the gradient, i.\,e. $\mathcal{O}(r^d(|\varOmega|+n)+r^{d+1})$.

\begin{remark}
	For the matrix case $d=2$, the Hessian expression \eqref{eqn:hessian} can be simplified to recover the expression shown in 
	\cite{vandereycken2013low,absil2013extrinsic},
	\begin{equation*}
		\begin{split}
			\Hess f(X) [\xi] = & \Proj_U \Proj_\varOmega(\xi) \Proj_V + 
			\Proj_U^\perp \big[ \Proj_\varOmega(\xi) + \Proj_\varOmega(X-A)\dot{V}\varSigma^{-1}V^\T \big] \Proj_V\\
			& + \Proj_U\big[ \Proj_\varOmega(\xi) + U \varSigma^{-1} \dot{U}^\T \Proj_\varOmega(X-A)\big] \Proj_V^\perp,
		\end{split}
	\end{equation*}
	where we identify the Tucker decomposition with the usual notation for the SVD, i.\,e. $U=U_1$, $V=U_2$ and $\varSigma=C$.
\end{remark}

\section{Riemannian models and trust-region methods} \label{sec:models}

In principle, the results of the previous subsections can be used to conceive a Riemannian Newton method for the solution of problem \eqref{eqn:problem}.
Such a method has been proposed in \cite[pp.\,279--283]{udriste1994convex}, where a convergence proof is given for strongly convex
functions \cite[Definition~1.1 in Chapter~7]{udriste1994convex}, using retraction by the exponential mapping (i.\,e. moving locally on a geodesic).
\cite[Theorem~4.4]{smith1994optimization} proves quadratic convergence of the method to a critical point. \cite[Theorem~6.3.2]{absil2008optimization}
provides a generalization for general retractions.

However, a plain Newton method has some well-known drawbacks:
\begin{enumerate}
	\item The convergence radius may be small, i.\,e. if the initial guess is too far from a critical point the method may diverge. \label{itm:drawback1}
	\item Each step requires the solution of a linear system. This may be expensive and conceptually difficult if the Hessian operator
	is not even given explicitly but in terms of the action on a vector in the tangent space, as in \eqref{eqn:hessian}. \label{itm:drawback2}
\end{enumerate}

There exists a number of strategies for remedying these problems. An intuitive method for globalizing the convergence of a Newton method is to modify
the Hessian such that the solution $\xi$ of
\begin{equation} \label{eqn:newton}
	\Hess f(x_k) [\xi] = - \grad f(x_k)
\end{equation}
defines a descent direction, see \cite[Section~3.4]{nocedal2006numerical} for an overview in the Euclidean case. In
\cite[Section~6.2]{absil2008optimization} a generalization to the Riemannian case is proposed, replacing the Newton equation with
\begin{equation*}
	\big( \Hess f(x_k) + E_k \big) [\xi] = - \grad f(x_k),
\end{equation*}
where $E_k$ is a sequence of positive-definite linear operators on the tangent spaces $T_{x_k}\Mf$.

However, such perturbed Newton methods rely on heuristics, and their general convergence properties are not well understood. Moreover, they still require
the solution of a linear system in each iteration. A way to circumvent this are trust-region methods \cite{conn2000trust}, which find a critical point of the 
function $f$ by minimizing a sequence of constraint quadratic models $m_{x_k}$. Our exposition follows the generalization to Riemannian optimization as given 
by Absil et al. \cite{absil2007trust}.

\subsection{Models on a Riemannian manifold $\Mf$} \label{subsec:models_mf}

For a real-valued function $f$ on a Riemannian manifold $\Mf$, a function $m_x$ is called an \textit{order-$q$ model}, $q>0$, of $\Mf$ in $x\in\Mf$ if there 
exists a neighbourhood $\mathcal{U}$ of $x$ in $\Mf$ and a constant $c>0$ such that
\begin{equation*}
	\big| f(y) - m_x(y) \big| \leq c \big( \dist(x,y) \big)^{q+1}, \text{\;\;\;for all\;\,} y \in \mathcal{U},
\end{equation*}
where $\dist$ denotes the Riemannian (geodesic) distance on $\Mf$. It can be shown \cite[Proposition~7.1.3]{absil2008optimization} that a model
$m_x$ is order-$q$ if and only if there exists a neighbourhood $\mathcal{U}'$ of $x$ in $\Mf$ and a constant $c'>0$ such that
\begin{equation*}
	\big| f(y) - m_x(y) \big| \leq c \big\| R_x^{-1}(y)\big\|^{q+1}, \text{\;\;\;for all\;\,} y \in \mathcal{U}.
\end{equation*}
i.\,e. the order of a model can be assessed using any retraction and we can avoid working with the exact geodesic.

Given a retraction $R$, this result allows to build a model for $f$ by simply taking a truncated Taylor expansion of
\begin{equation*}
	\widehat{f}_x \coloneqq f \circ R_x,
\end{equation*}
for any $x \in \Mf$. The definition of $\widehat{f}_x:T_x\Mf\rightarrow\R$ as a real-valued function on a Euclidean space
allows us to use standard results from multivariate analysis. A simple first-order model is then given by
\begin{equation*}
	\widehat{m}_x = \widehat{f}_x(0_x) + \Deriv \widehat{f}_x(0_x)[\xi] = f(x) + \dotprod{\grad f(x)}{\xi},
\end{equation*}
where the second equality follows form the rigidity condition of the retraction. A generic second-order model is given by
\begin{align*}
	\widehat{m}_x &= \widehat{f}_x(0_x) + \Deriv \widehat{f}_x(0_x)[\xi] + \tfrac{1}{2} \Deriv^2 \widehat{f}_x(0_x)[\xi,\xi]\\
	&= f(x) + \dotprod{\grad f(x)}{\xi} + \tfrac{1}{2} \dotprodb{\Hess \widehat{f}(x)[\xi]}{\xi}.
\end{align*}
A straightforward and useful modification is obtained by replacing the Euclidean Hessian on the tangent space $\Hess \widehat{f}(x)$ by the Riemannian expression
$\Hess f(x)$. The following lemma shows that this can be done in a critical point of $f$ without any loss of information.

\begin{lemma}  \label{lem:model_critp}
	\cite[Proposition~5.5.6]{absil2008optimization} Let $R$ be a retraction and let $x^*$ be a critical point of a real-valued function $f$ on $\Mf$, i.\,e.
	$\grad f(x^*)=0_{x^*}$. Then
	\begin{equation*}
		\Hess f(x^*) = \Hess \widehat{f}(0_{x^*}).
	\end{equation*}
\end{lemma}

Thus, we can define a model
\begin{equation*}
	m_x = f(x) + \dotprod{\grad f(x)}{\xi} + \tfrac{1}{2} \dotprodb{\Hess f(x)[\xi]}{\xi},
\end{equation*}
which does not make any use of a retraction. However, Lemma~\ref{lem:model_critp} only guarantees that $m_x$ matches
$f$ up to second order if $x$ is a critical point. In general, we can only prove that it will only give us a first-order model. The model
$m_x$ can be shown to be of second order for general $x$ if the retraction $R$ is of second order, i.\,e. if it
preserves second-order information of the exponential map, cf. \cite[Proposition~5.5.5]{absil2008optimization}.
However, numerical results presented later in this section suggest that in our case the result also holds for general points on the manifold.

\subsection{Models of different orders on $\Mfr$} \label{subsec:models_mfr}

\begin{figure}
	\begin{tabular}{ccc}
		\includegraphics[scale=0.4]{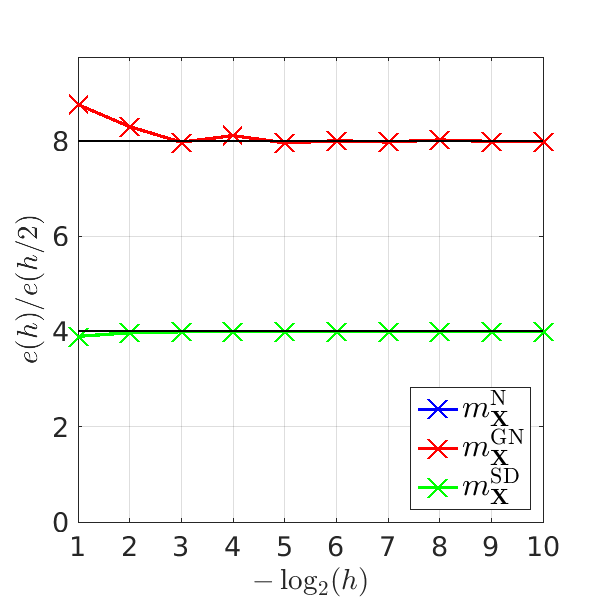}&
		\includegraphics[scale=0.4]{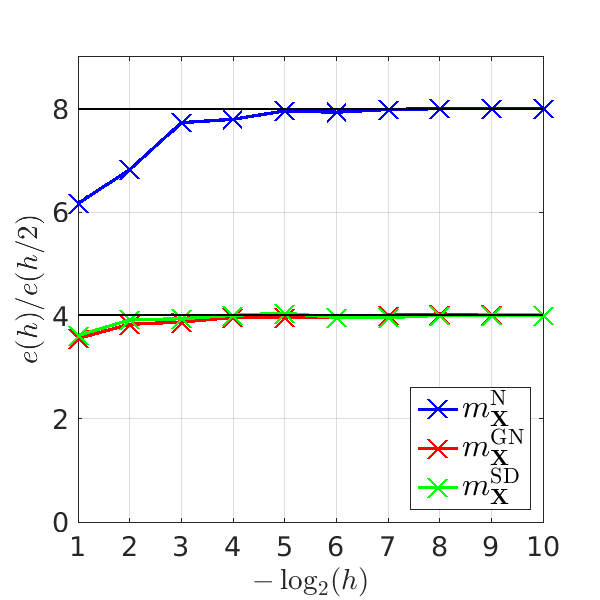}\\
		$f(\mathbf{X})=0$&
		$f(\mathbf{X})=1.2\times10$\\
		\includegraphics[scale=0.4]{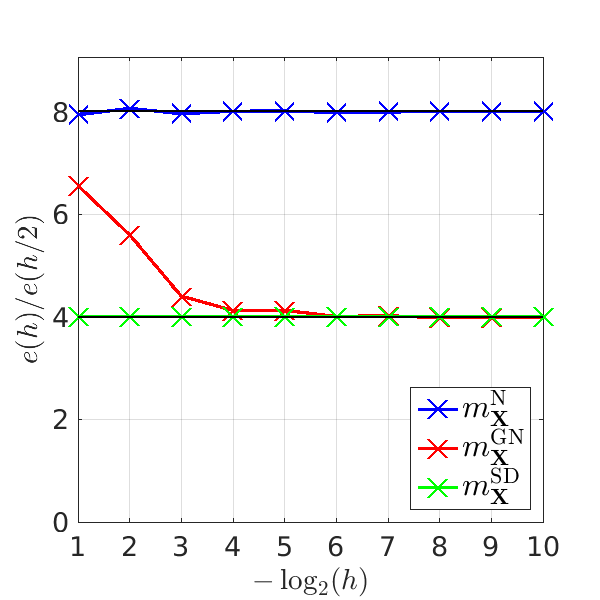}&
		\includegraphics[scale=0.4]{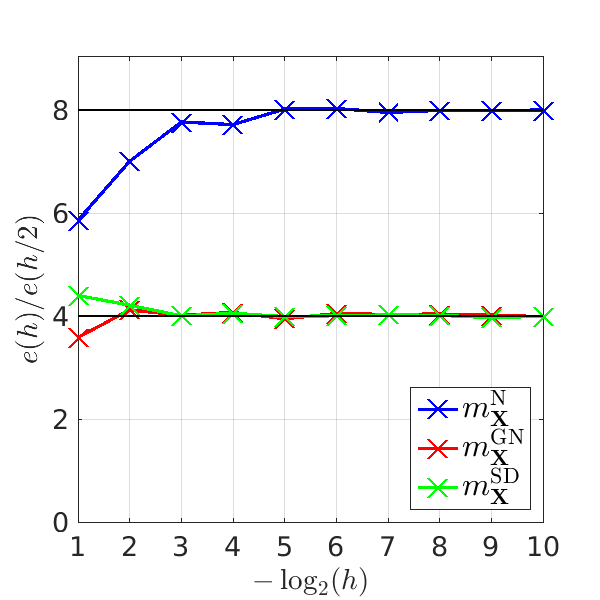}\\
		$f(\mathbf{X})=2.1\times10^{-2}$&
		$f(\mathbf{X})=1.9\times10^{2}$\\
		\includegraphics[scale=0.4]{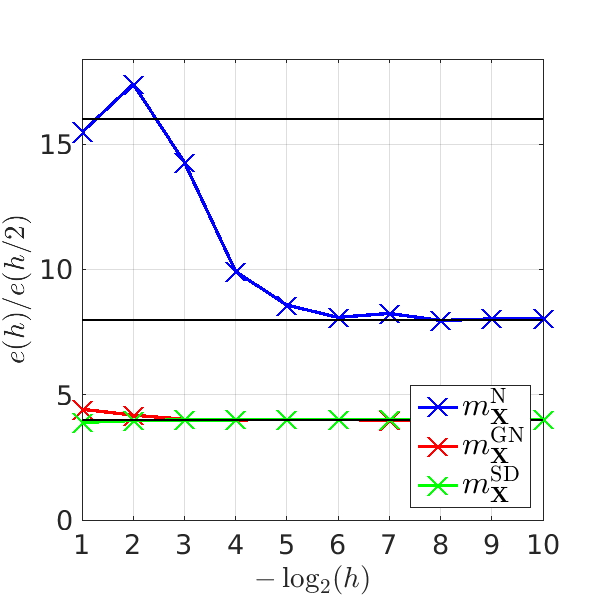}&
		\includegraphics[scale=0.4]{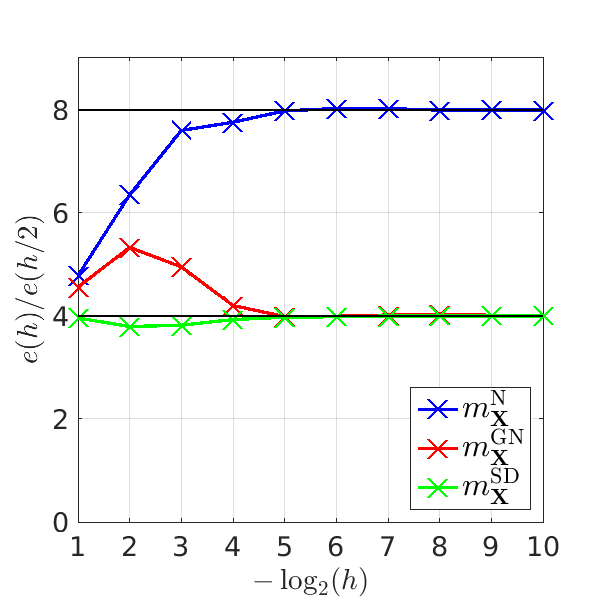}\\
		$f(\mathbf{X})=3.6\times10^{2}$&
		$f(\mathbf{X})=1.8\times10^{3}$\\
	\end{tabular}
	\caption{The unknown tensor $\mathbf{A}$ has full rank, i.\,e. $\mathbf{A}\notin\Mfr$.}
	\label{fig:model_full}
\end{figure}

\begin{figure}
	\begin{tabular}{ccc}
		\includegraphics[scale=0.4]{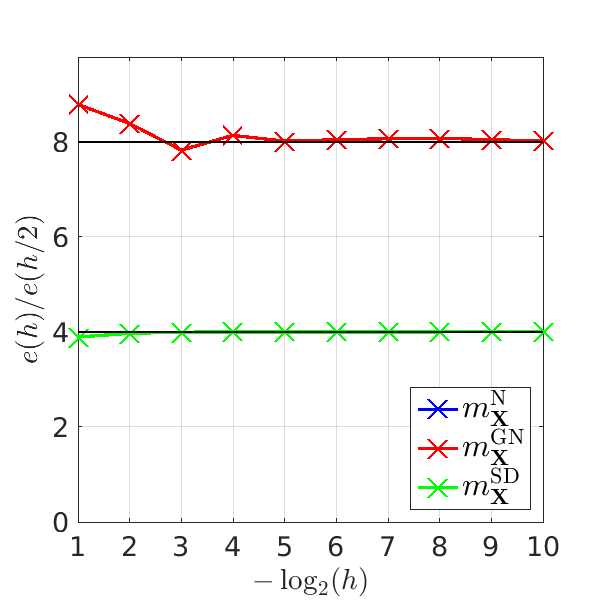}&
		\includegraphics[scale=0.4]{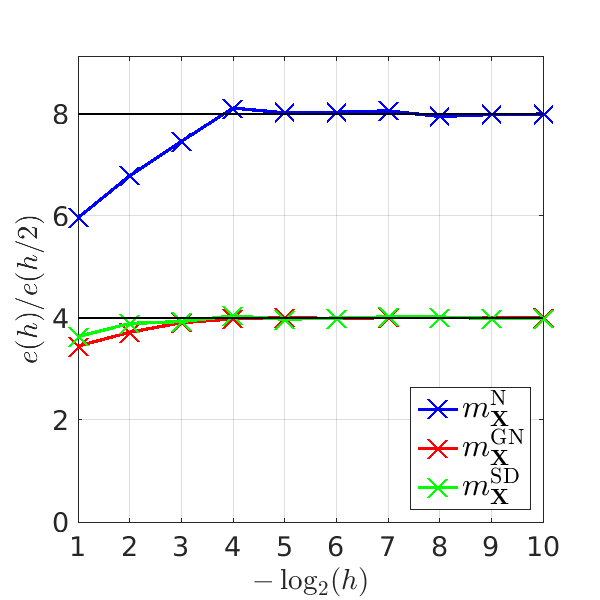}\\
		$f(\mathbf{X})=0$&
		$f(\mathbf{X})=1.3\times10$\\
		\includegraphics[scale=0.4]{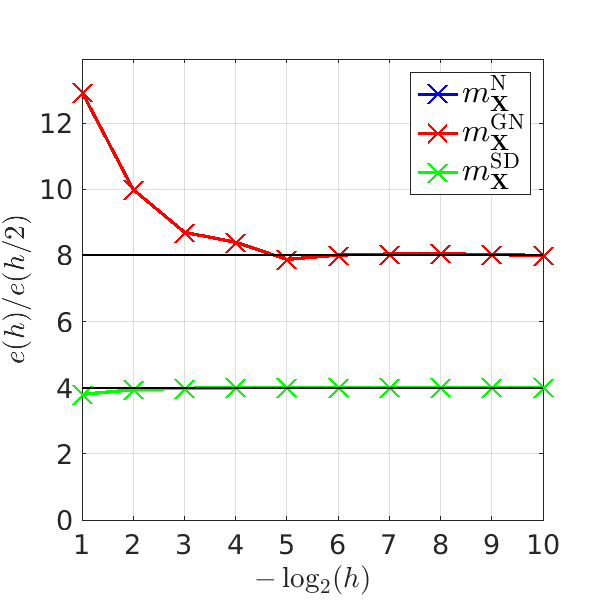}&
		\includegraphics[scale=0.4]{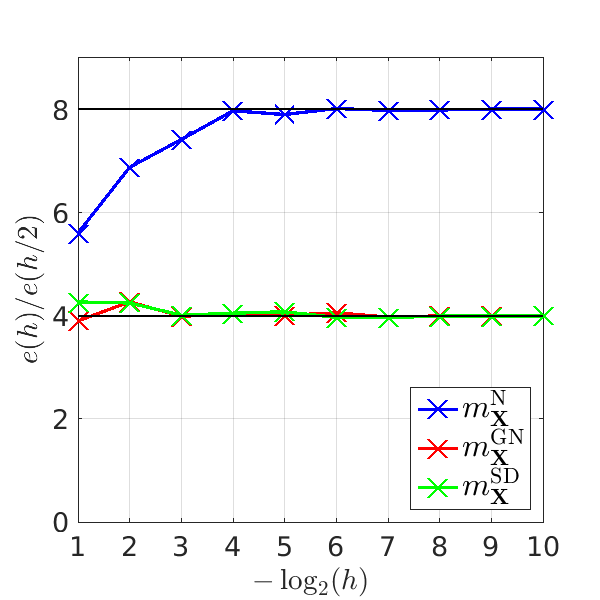}\\
		$f(\mathbf{X})=0$&
		$f(\mathbf{X})=1.5\times10^{2}$\\
		\includegraphics[scale=0.4]{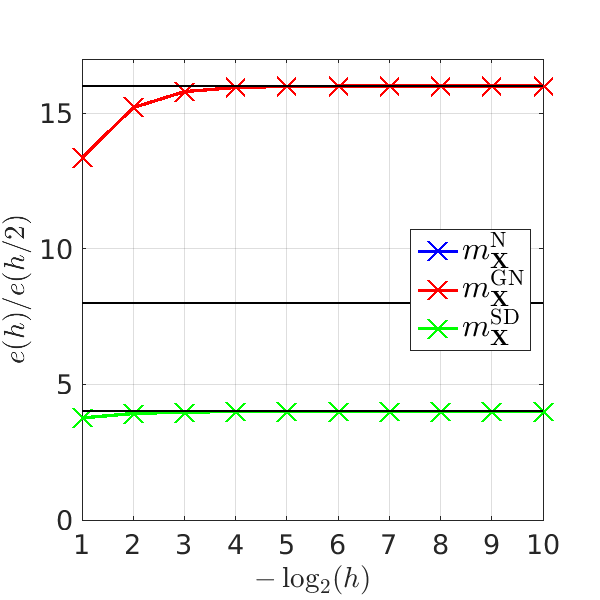}&
		\includegraphics[scale=0.4]{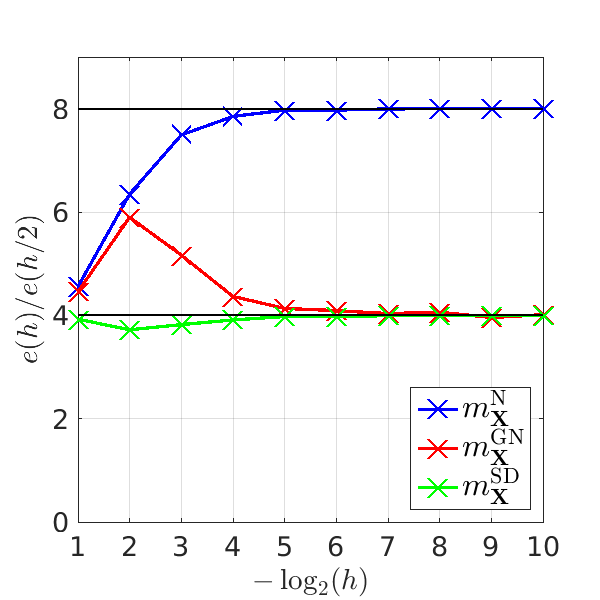}\\
		$f(\mathbf{X})=0$&
		$f(\mathbf{X})=1.4\times10^{3}$\\
	\end{tabular}
	\caption{The unknown tensor $\mathbf{A}$ has low rank, i.\,e. $\mathbf{A}\in\Mfr$.}
	\label{fig:model_low}
\end{figure}

We consider the manifold $\Mfr$ of fixed-rank tensors and would like to assess the quality of different model functions.
In accordance with the previous subsection, we consider a first-order model
\begin{equation} \label{eqn:model_sd}
	m_\mathbf{X}^\mathrm{SD}(\xi) \coloneqq f(\mathbf{X}) + \dotprod{\grad f(\mathbf{X})}{\xi},
\end{equation}
where the superscript indicates that this model corresponds to a steepest-descent method, and a second-order model
\begin{equation*}
	m_\mathbf{X}^\mathrm{N}(\xi) \coloneqq f(\mathbf{X}) + \dotprod{\grad f(\mathbf{X})}{\xi}
	+ \tfrac{1}{2} \dotprodb{\Hess f(\mathbf{X})[\xi]}{\xi},
\end{equation*}
where the superscript indicates that this model corresponds to a Newton method. Furthermore, we would like to assess
the quality of a Hessian approximation which drops the curvature term in Corollary~\ref{cor:curvature_term} and thus
ignores the second-order geometry of $\Mfr$. This is given by omitting the second term in \eqref{eqn:hessian}, and
just considering the projection of the Euclidean Hessian i.\,e.
\begin{equation} \label{eqn:model_n}
	\widetilde{\Hess} f(\mathbf{X}) [\xi] = \Proj_\varOmega \xi.
\end{equation}
Omitting the curvature term of the Hessian corresponds to a Riemannian Gauß--Newton method, as described in \cite[Subsection~8.4.1]{absil2008optimization}. Thus, we consider the model function
\begin{equation} \label{eqn:model_gn}
	m_\mathbf{X}^\mathrm{GN}(\xi) \coloneqq f(\mathbf{X}) + \dotprod{\grad f(\mathbf{X})}{\xi}
	+ \tfrac{1}{2} \dotprodb{\widetilde{\Hess} f(\mathbf{X})[\xi]}{\xi}.
\end{equation}

As usual, we can expect a Gauß--Newton method to converge superlinearly (as the corresponding model to be of order higher
than $1$) if the residual of the least-squares problem is low. This can be seen in terms of \eqref{eqn:hessian}, where the
curvature term is given as
\begin{equation*}
	\big( \Hess f(\mathbf{X}) - \widetilde{\Hess} f(\mathbf{X}) \big) [\xi] = 
	\Proj_\mathbf{X} \Deriv_\xi\Proj_\mathbf{X} \Proj_\mathbf{X}^\perp(\Proj_\varOmega \mathbf{X} - \Proj_\varOmega \mathbf{A}),
\end{equation*}
which is clearly equal to zero if $\Proj_\varOmega \mathbf{X} = \Proj_\varOmega \mathbf{A}$ and hence $f(\mathbf{X})=0$.

To assess the order of a model, we define for a given $m_\mathbf{X}$ the \textit{model error}
\begin{equation*}
	e(\xi,h) \coloneqq \big| \widehat{f}_\mathbf{X}(h\xi) - m_\mathbf{X}(h\xi) \big|,
\end{equation*}
for $\xi \in T_\mathbf{X}\Mfr$ and $h\geq 0$. Then  $m_\mathbf{X}$ is an order-$q$ model  in
$\mathbf{X}$ if and 
only if
\begin{equation*}
	e(\xi,h) = \mathcal{O}\big(h^{q+1}\big), \text{\;\;\;for all\;\,} \xi \in T_\mathbf{X}\Mfr.
\end{equation*}

In Figures~\ref{fig:model_full}~and~\ref{fig:model_low}, we test the model orders of
\eqref{eqn:model_sd}--\eqref{eqn:model_gn}. We generate random tensors
$\mathbf{B}_1,\dotsc,\mathbf{B}_{1000}\in \R^{10\times 10\times 10}$ with normally distributed entries and
project them onto a given tangent space of $\Mf_{(3,3,3)}$ to get $\xi_i = \Proj_\mathbf{X}(\mathbf{B}_i)$. We normalize the resulting
vectors to get $\norm{\xi_i} = 1$. We compute the errors $e(\xi_i,2^{-j})$ for $j = 0,\cdots,10$, and plot the geometric
mean of the factors $(\xi_i,2^{-(j+1})/(\xi_i,2^{-j})$ over all $i$. The first columns contain the results for
a stationary point of $f$, i.\,e. $\norm{\grad f({\mathbf{X^*}})} = 0$, the second columns contain the results for an
arbitrary point on the manifold with $\norm{\grad f(\mathbf{X})} \neq 0$. The first, second and third rows
contain results for different sampling sizes, with $| \varOmega| = 10,100,1000$, respectively. Note that
$|\varOmega| = 1000$ represents full sampling, i.\,e. vector approximation. We write $f(\mathbf{X})=0$ whenever the
function value computed is smaller that the machine precision of $10^{-16}$.

We observe that the model function $m_\mathbf{X}^\mathrm{SD}$, indeed, provides results of first order in all cases.
The model function $m_\mathbf{X}^\mathrm{N}$ provides results of second order not only in critical points, as has been
proved by theory, but also in general points on the manifold. This can be seen as an indication that the retraction by
HOSVD preserves second-order information although we cannot prove this. We also observe that the Gauß--Newton type model
function $m_\mathbf{X}^\mathrm{GN}$ gives second-order results whenever the curvature term is small enough, otherwise
it is only a first-order model; this matches the theoretical predictions we made earlier. It is especially worth noting that, for
$\mathbf{A}\in\Mfr$, a Gauß--Newton model is sufficient; however, this result is not robust if we add some noise. Note that in the cases where 
the blue curve cannot be seen in the plot, the models $m_\mathbf{X}^\mathrm{GN}$ and $m_\mathbf{X}^\mathrm{N}$ match almost
exactly.

We also remark that in the case of exact tensor reconstruction, i.\,e. $\mathbf{A}\in\Mfr$ and ${| \varOmega| = \prod_i n_i}$ (the lower-left
plot in Figure~\ref{fig:model_low}), both $m_\mathbf{X}^\mathrm{N}$ and $m_\mathbf{X}^\mathrm{GN}$ seem to be models of order $3$, which means that the third-order term in the Taylor expansion of $f$ vanishes. This may be attributed to a possible symmetry of $f$ around the local minimizer $\mathbf{X}^*=\mathbf{A}$
in this case, i.\,e. $f(\Exp_{\mathbf{X}^*}(\xi))=f(\Exp_{\mathbf{X}^*}(-\xi))$, where $\Exp$ denotes the exponential map. This means that the
odd-exponent terms in the Taylor expansion are equal to zero.
However, we cannot verify this theoretically as we do not have a closed-form expression for the exponential map on $\Mfr$.

\subsection{Riemannian trust-region method} \label{subsec:rtr}

The main idea of trust-region methods is solving a model problem
\begin{equation} \label{eqn:model_problem}
	\begin{split}
		&\min_{\eta\in T_{\mathbf{X}_k}\Mfr} m_{\mathbf{X}_k}(\eta)\\
		\text{s.\,t.} ~ &\norm{\eta} \leq \varDelta_k,
	\end{split}
\end{equation}
for some $\varDelta_k\geq 0$ in each iteration $k$ to obtain a search direction $\eta_k$. To get meaningful results
it is crucial to check how well the model $m_{\mathbf{X}_k}$ approximates $\widehat{f}$ in $T_{\mathbf{X}_k\Mfr}$ in the 
neighbourhood of $0_{\mathbf{X}_k}$. This can be expressed in the form of the quotient
\begin{equation} \label{eqn:rho}
	\rho_k \coloneqq \frac{\widehat{f}(0_{\mathbf{X}_k})-\widehat{f}(\eta_k)}
	{m_{\mathbf{X}_k}(0_{\mathbf{X}_k})-m_{\mathbf{X}_k}(\eta_k)}.
\end{equation}
If $\rho_k$ is small (convergence theory suggests that $\rho'<\tfrac{1}{4}$ is an appropriate threshold), then the model is very inaccurate: the step must be rejected, and the trust-region radius
$\varDelta_k$ must be reduced. If $\rho_k$ is small but less dramatically so, then the step is accepted but the trust-region 
radius is reduced. If $\rho_k$ is close to 1, then there is a good agreement between the model and the function over the 
step, and the trust-region radius can be expanded. If $\rho_k\gg 1$, then the model is inaccurate, but the overall 
optimization iteration is producing a significant decrease in the cost. If this is the case and the restriction in
\eqref{eqn:model_problem} is active, we can try to expand the trust-region radius as long as we stay below a predefined
bound $\bar{\varDelta}>0$. This method is summarized in Algorithm~\ref{alg:rtr}, cf. \cite[Algorithm~1]{absil2007trust}.

\begin{algorithm}
	\caption{Riemannian trust-region method for $\Mfr$} \label{alg:rtr}
	\begin{algorithmic}[1]
		\Require{Initial iterate $x_0 \in \Mf$;
		parameters $\bar{\varDelta}>0$, $\varDelta_0 \in(0,\bar{\varDelta})$, $\rho' \in (0,\tfrac{1}{4})$.}
			\For{$k=0$ \textbf{until} convergence}
				\State Obtain $\eta_k$ by approximately solving \eqref{eqn:model_problem}
				\State $<$Test for convergence$>$
				\State Evaluate $\rho_k$ from \eqref{eqn:rho}
				\If {$\rho_k<\tfrac{1}{4}$}
					\State $\varDelta_{k+1} = \tfrac{1}{4}\varDelta_k$
				\ElsIf {$\rho_k > \tfrac{3}{4}$ and $\norm{\eta_k} = \varDelta_k$}
					\State $\varDelta_{k+1} = \min(2\varDelta_k,\bar{\varDelta})$
				\Else
					\State $\varDelta_{k+1} = \varDelta_k$
				\EndIf
				\If {$\rho_k>\rho'$}
					\State $\mathbf{X}_{k+1} = R_{\mathbf{X}_k}(\eta_k)$
				\Else
					\State $\mathbf{X}_{k+1} = \mathbf{X}_k$
				\EndIf
			\EndFor
	\end{algorithmic}
\end{algorithm}

There exist different strategies for (approximately) solving the trust-region subproblems \eqref{eqn:model_problem}. We apply a truncated CG
method \cite[Algorithm~2]{absil2007trust}, which is a straightforward adaptation of Steighaug's method \cite{steihaug1983conjugate} for problems in $\R^n$. It ensures that the CG method is stopped after a fixed maximal
number of iterations $K_{\max}$. Since a CG iteration just requires a fixed number of matrix-vector products, the total
cost of the trust-region method with exact Hessian evaluation is given by
\begin{equation*}
	\mathcal{O}\big(K_{\max}(r^d(|\varOmega|+n)+r^{d+1})\big)
\end{equation*}

The convergence theory follows standard techniques from Euclidean optimization \cite{conn2000trust}. Under some technical assumptions, it can be shown
that Algorithm~\ref{alg:rtr} converges globally to a stationary point \cite[Theorem~4.4]{absil2007trust} of $f$. Locally superlinear convergence to a
nondegenerate local minimum can be shown
\cite[Theorem~4.12]{absil2007trust} as long as the quadratic term in $m_{\mathbf{X}_k}$ is a sufficiently good Hessian approximation of $f$.

In general, we cannot rule out Algorithm~\ref{alg:rtr} converging to a nonregular minimum if
${|\varOmega|<\dim(\Mfr)}$. If this causes problems, we can enforce positive-definiteness
of the Hessian by considering a cost function regularized with an identity term
\begin{equation*}
	f_\mu(\mathbf{X}) = \frac{1}{2} \normb{\Proj_\varOmega\mathbf{X}-\Proj_\varOmega\mathbf{A}}^2
	+ \frac{\mu}{2}\norm{\mathbf{X}}^2,
\end{equation*}
for some $\mu>0$. However, such a problem may not be well-posed since there is not enough information provided to recover
$\mathbf{X}$ in a meaningful way. Moreover, in our practical experiments we did not have a need to use this regularization.

\section{Numerical experiments} \label{sec:experiments}

We implemented our method in \textsc{Matlab} version 2015b using the Tensor Toolbox version 2.6 \cite{bader2006algorithm,bader2015matlab} for the basic tensor arithmetic 
and Manopt version 3.0 \cite{boumal2014manopt} for handling the Riemannian trust-region scheme. All tests were performed on a quad-core Intel i7-2600 CPU with 8 GB
of RAM running 64-Bit Ubuntu 16.04 Linux.

In Algorithm~\ref{alg:rtr}, we choose the standard parameters $\bar{\varDelta}=\dim(\Mfr)$, $\varDelta_0=\bar{\varDelta}/8$, $\rho'=0.1$.
The initial guess $\mathbf{X}_0$ is generated randomly by a uniform distribution on $(0,1)$ for each entry in the factors in the Tucker decomposition.
We apply a QR factorization in each mode to ensure that the basis matrices are orthogonal. The sampling set $\varOmega$ is chosen from a uniform distribution
on the index set.

\subsection{Uniformly distributed random data} \label{subsec:uniform_data}

\begin{figure}
	\begin{tabular}{ccc}
		\includegraphics[scale=0.4]{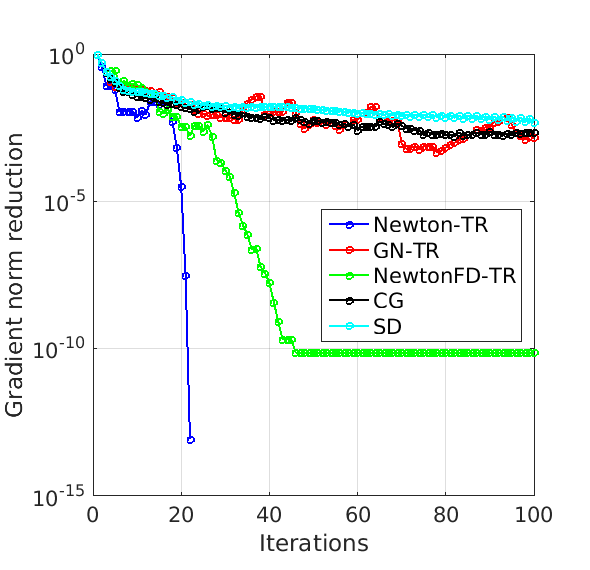}&
		\includegraphics[scale=0.4]{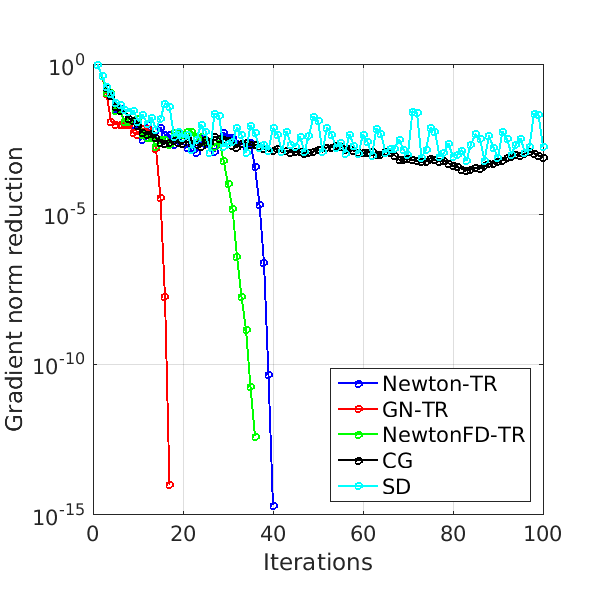}\\
		$|\varOmega| = 0.05\times\prod_i n_i$, $\mathbf{A}\notin\Mfr$&
		$|\varOmega| = 0.05\times\prod_i n_i$, $\mathbf{A}\in\Mfr$\\
		\includegraphics[scale=0.4]{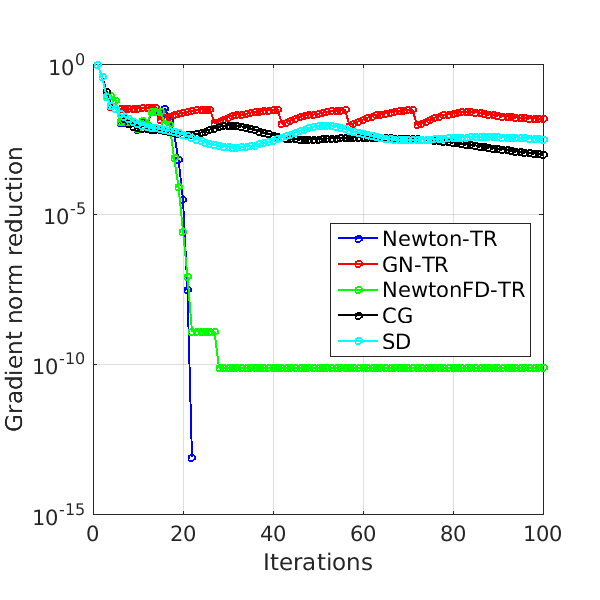}&
		\includegraphics[scale=0.4]{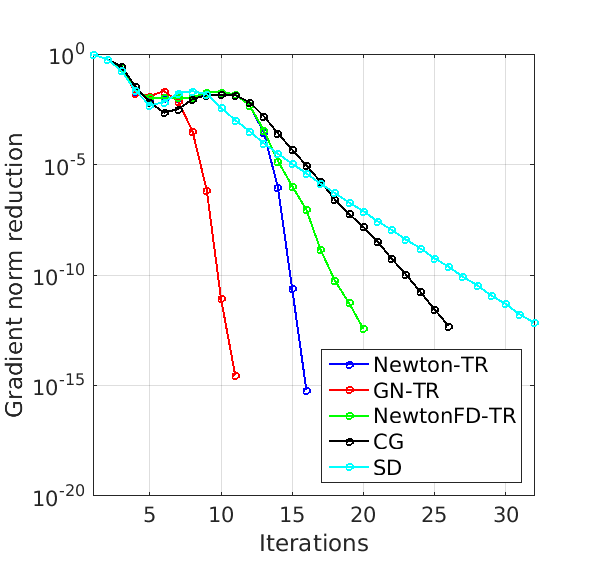}\\
		$|\varOmega| = 0.5\times\prod_i n_i$, $\mathbf{A}\notin\Mfr$&
		$|\varOmega| = 0.5\times\prod_i n_i$, $\mathbf{A}\in\Mfr$
	\end{tabular}
	\caption{Convergence of Riemannian methods for \eqref{eqn:problem} with $n_i \equiv 20$ and $r_i \equiv 2$.}
	\label{fig:convergence}
\end{figure}

We test the convergence behaviour of Algorithm~\ref{alg:rtr} for the recovery of a partially known tensor $\mathbf{A}$ with uniformly distributed entries. We observe that the trust-region method with exact Hessian computation yields
superlinear convergence after a small number of iterations in all cases observed here. The finite difference Hessian
approximation shows similar behaviour, however, the convergence is slower and becomes less reliable for a large
gradient norm reduction. The Gauß--Newton Hessian approximation shows shows superlinear convergence behaviour if
$\mathbf{A}\in\Mfr$, but not in the case $\mathbf{A}\notin\Mfr$, as predicted in the previous sections. The state-of-the-art
Riemannian method, nonlinear CG \cite{kressner2014low}, shows linear convergence with convergence rate superior to steepest
descent, but the convergence rate may slow, especially in the case of noise.

\subsection{Survey data} \label{subsec:survey_data}

\begin{figure}
	\begin{tabular}{ccc}
		\includegraphics[scale=0.4]{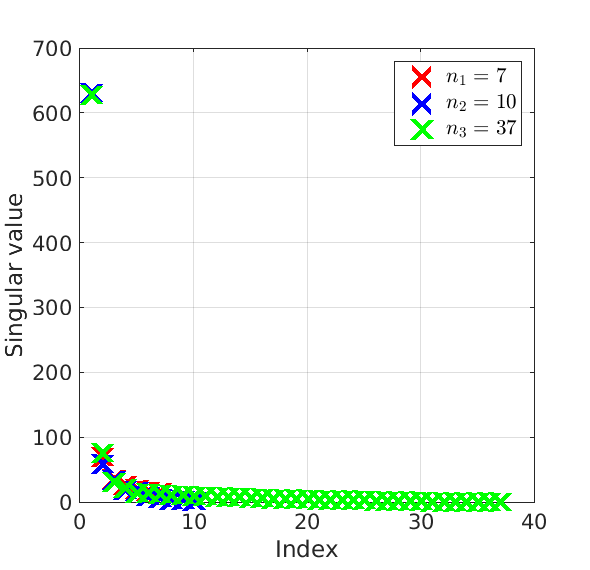}&
		\includegraphics[scale=0.4]{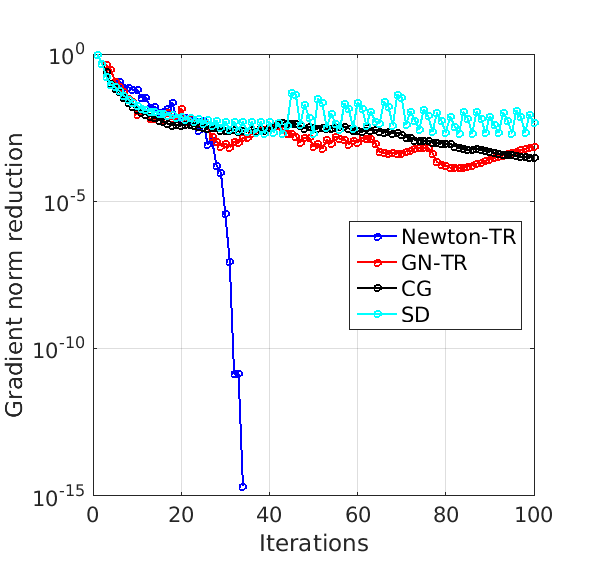}
	\end{tabular}
	\caption{Left: Singular values of the data set \cite{kroonenberg2017information}; right: convergence of Riemannian
	methods for tensor completion with $|\varOmega| = 0.5\times\prod_i n_i$ and $\mathbf{r}=(3,5,5)$.}
	\label{fig:bus}
\end{figure}

In survey statistics, data in the form of order-$3$ tensors arises in a natural way: for $n_1$ of individuals, $n_2$ 
properties are collected over $n_3$ time points; see, for example, \cite{kroonenberg1983three}. We choose a standard data 
set \cite{kroonenberg2017information}, containing reading proficiency test measures of schoolchildren over a period of time. 
A typical problem in such data sets in practice is missing entries, resulting from nonresponse or failure to enter some of 
the data points correctly; see \cite{little2014statistical}. A typical application case is a sampling set greater or equal
to haf the total tensor size. As Figure~\ref{fig:bus}
shows, data of this type shows rapidly decaying singular values, especially in the time mode ($i=3$) and our trust-region
method can be used to retrieve deleted data in a low-rank framework. The trust-region method also converges superlinearly
in this case. The simplified Gauß--Newton trust-region scheme does not show superlinear convergence since noise is present
in this application case. The
trust-region methods also compares favorably with nonlinear CG in this case. Our results can be seen as an indication that
Riemannian trust-region methods can be used for statistical data recovery.

\section{Conclusions and discussion} \label{sec:conclusion}

We have derived the Riemannian Hessian for functions on the manifold of tensors of fixed multilinear rank in Tucker format. We have shown
that it can be used to construct a rapidly and robustly converging trust-region scheme for tensor completion. Furthermore,
this is the first theoretical result on the second-order properties of the given manifold; we believe this to be useful
for an improved understanding of the underlying geometry. Our numerical results also indicate that Riemannian optimization
is a suitable technique for the recovery of missing entries from multilinear survey data with low-rank structure.
We believe that this aspect merits further exploration; a comparison of Riemannian techniques with standard imputation
methods from statistics \cite{little2014statistical} may reveal opportunities and limitations of this approach. For this,
a better understanding of the sensitivity of the Tucker decomposition to perturbations is required.

Another well-known way to obtain superlinear convergence is a Riemannian BFGS method. In recent research, several
schemes have been proposed, generalizing this standard method from Euclidean optimization to the Riemannian case; see
\cite[Subsection~5.2]{huang2016intrinsic} for an application to the manifold of matrices of fixed rank. Extending this
idea to tensors merits some examination. For high-dimensional applications with $d\gg 3$, hierarchical tensor formats
\cite{uschmajew2013geometry,grasedyck2013literature} are crucial; see \cite{silva2015optimization} for a Riemannian 
optimization approach.

\section*{Acknowledgements}

The authors thank Lars Grasedyck and Bart Vandereycken for fruitful discussions. Jan Pablo Burgard pointed out the applicability of this work to data from survey statistics in general and to the data set \cite{kroonenberg2017information} in particular. The first author has been supported by the German 
Research Foundation (DFG) within the Research Training Group 2126: `Algorithmic Optimization'.

\bibliographystyle{wileyj}
\bibliography{heideltensor}

\end{document}